\documentclass[12pt]{amsart}
\usepackage{amsmath,amssymb,amsthm,amscd,graphicx,url}
\newcommand{\Z}{\mathbb{Z}}
\newcommand{\Q}{\mathbb{Q}}

\newcommand{\N}{\mathbb{N}}

\newcommand{\legen}[2]{\genfrac{(}{)}{}{}{#1}{#2}}
\newcommand{\ord}{{\rm ord}}

\newtheorem{thm}{Theorem}[section]
\newtheorem*{thm*}{Theorem}
\newtheorem{cor}[thm]{Corollary}
\newtheorem{conjec}[thm]{Conjecture}
\newtheorem{lem}[thm]{Lemma}
\newtheorem*{ex}{Example}
\newtheorem*{rem}{Remark}

\begin{document}

\title{When is $a^{n} + 1$ the sum of two squares?}


\author[Dresden]{Greg Dresden}
\address{Department of Mathematics, Washington \& Lee University, Lexington, VA 24450}
\email{dresdeng@wlu.edu}

\author[Hess]{Kylie Hess}
\address{Department of Mathematics, Rose-Hulman Institute of Technology, Terre Haute, IN 47803}
\email{hessko@rose-hulman.edu}

\author[Islam]{Saimon Islam}
\address{Department of Mathematics, Washington \& Lee University, Lexington, VA 24450}
\email{islams19@wlu.edu}

\author[Rouse]{Jeremy Rouse}
\address{Department of Mathematics and Statistics, Wake Forest University, Winston-Salem, NC 27109}
\email{rouseja@wfu.edu}

\author[Schmitt]{Aaron Schmitt}
\address{Department of Mathematics, Washington \& Lee University, Lexington, VA 24450}
\email{schmitta18@wlu.edu}

\author[Stamm]{Emily Stamm}
\address{Department of Mathematics and Statistics, Vassar College, Poughkeepsie, NY 12604}
\email{emstamm@vassar.edu}

\author[Warren]{Terrin Warren}
\address{Department of Mathematics, University of Georgia, Athens, GA 30602}
\email{warrentm@uga.edu}

\author[Yue]{Pan Yue}
\address{Department of Mathematics, Washington \& Lee University, Lexington, VA 24450}
\email{pany19@wlu.edu}

\thanks{The second, fourth, sixth and seventh authors were supported by the NSF grant DMS-1461189.}

\subjclass[2010]{Primary 11E25; Secondary 11C08, 11R18}

\begin{abstract}
  Using Fermat's two squares theorem and properties of cyclotomic
  polynomials, we prove assertions about when numbers of the form
  $a^{n}+1$ can be expressed as the sum of two integer squares. We
  prove that $a^n + 1$ is the sum of two squares for all $n \in \N$ if
  and only if $a$ is a perfect square. We also prove that for
  $a\equiv 0,1,2\pmod{4},$ if $a^{n} + 1$ is the sum of two squares,
  then $a^{\delta} + 1$ is the sum of two squares for all
  $\delta | n, \ \delta>1$. Using Aurifeuillian factorization, we show
  that if $a$ is a prime and $a\equiv 1 \pmod{4}$, then there are
  either zero or infinitely many odd $n$ such that $a^n+1$ is the sum
  of two squares. When $a\equiv 3\pmod{4},$ we define $m$ to be the
  least positive integer such that $\frac{a+1}{m}$ is the sum of two
  squares, and prove that if $a^n+1$ is the sum of two squares for any
  odd integer $n,$ then $m | n$, and both $a^m+1$ and $\frac{n}{m}$ are 
sums of two squares.
\end{abstract}

\maketitle


\section{Introduction}


Many facets of number theory revolve around investigating terms of a
sequence that are \emph{interesting}. For example, if
$a_{n} = 2^{n} - 1$ is prime (called a Mersenne prime), then $n$
itself must be prime (Theorem 18 of \cite[p. 15]{HardyAndWright}). In
this case, the property that is interesting is primality. Ramanujan
was interested in the terms of the sequence $a_{n} = 2^{n} - 7$ that
are perfect squares.  He conjectured that the only such terms are
those with $n = 3, 4, 5, 7$ and $15$, and Nagell proved this in 1948
(see \cite{Nagell}; a modern reference is
\cite[p. 96]{StewartTall}). Finally, if the Fibonacci sequence is
defined by $F_{0} = 0$, $F_{1} = 1$ and $F_{n} = F_{n-1} + F_{n-2}$
for $n \geq 2$, then $F_{n}$ is prime if and only if $n$ is prime or
$n = 4$ (Theorem 179 of \cite[p. 148]{HardyAndWright}), and the only
perfect powers in the Fibonacci sequence are $0$, $1$, $8$ and $144$,
which was proven by Bugeaud, Mignotte, and Siksek \cite{BMS} in 2006
using similar tools to the proof of Fermat's Last Theorem.

In this paper, we will consider a number 
to be \emph{interesting} if it can be expressed as the sum of
two squares. The earliest work on this topic 
relates to Pythagorean triples, which are integer solutions to $a^{2} + b^{2} = c^{2}$. Euclid supplied an infinite family of solutions: $a = m^{2} - n^{2}$,
$b = 2mn$ and $c = m^{2} + n^{2}$.

Fermat's two squares theorem classifies which numbers can be
written as the sum of two squares. Fermat claimed to have proven this theorem in his 1640 letter to Mersenne, but never shared the proof. The first published proof is attributed to Euler and was completed in 1749 (see \cite[p. 11]{Cox}).

\begin{thm*}[Fermat's two squares theorem]
A positive integer $N$ can be written as the sum of two squares if and only if in the prime factorization of $N$, 
\[
  N = \prod_{i=1}^{k} p_{i}^{e_{i}},
\]
we have $p_{i} \equiv 3 \pmod{4}$ if and only if $e_{i}$ is even.
\end{thm*}

In light of Fermat's theorem, integers that can be expressed as the sum of two squares become increasingly rare.
In particular, if $S(x)$ denotes the number of integers $n \leq x$ that are expressible as a sum of two squares,
then Landau proved \cite{Landau} in 1908 that
\[
  \lim_{x \to \infty} \frac{S(x)}{x / \sqrt{\ln(x)}} = K \approx 0.764.
\]
This can be stated more colloquially as ``the probability that a number $n$ is the sum of two squares is
$\frac{K}{\sqrt{\ln(n)}}$.''

We are interested in which terms in sequences of the form $a^{n} + 1$
can be written as a sum of two squares. In \cite{Curtis}, Curtis
showed that $2^{n}+1$ is the sum of two squares if and only if $n$ is
even or $n = 3$. Additionally, if $n$ is odd and $3^{n}+1$ is the sum
of two squares then $n$ must be the sum of two squares, and $3^{p}+1$
is the sum of two squares for all prime numbers $p|n$.

The focus of the present paper is to say as much as possible about
when $a^{n} + 1$ is the sum of two squares for a general positive
integer $a$.  This paper is the result of two undergraduate research
teams working simultaneously and independently over two months in the
summer of 2016.  The first team, from Wake Forest University,
consisted of students Hess, Stamm, and Warren, and was led by Jeremy
Rouse; the second team, from Washington \& Lee University, consisted
of students Islam, Schmitt, and Yue, and was led by Greg Dresden.
Remarkably, the two teams ended up covering many of the same
topics. Some of the results are unique to the Wake Forest team, while other
results were proved by both teams using different
methods. We carefully assign credit to the theorems in the first section by 
using the tags WF and W\&L in each result, with remarks as necessary.

In the case that $n = 2k$ is even, then
$a^{n} + 1 = \left(a^{k}\right)^{2} + 1^{2}$ is trivially the sum of
two squares. For this reason, we focus on cases when $n$ is odd. Our
first result is the following. 


\begin{thm}[WF]
\label{perfect}
The number $a^{n}+1$ is the sum of two squares for every $n \in \N$ if and only if $a$ is a perfect square.
\end{thm}
\begin{ex}
\end{ex}
\begin{enumerate}
    \item If $a=9,$ then $9^n+1=(3^{n})^2+1^2.$ 
    \item If $a=7,$ then there is some odd $n$ such that $7^n+1$ is not the sum of two squares. For example, $7^3+1$ is not the sum of two squares.
\end{enumerate}

Our next result gives specific criteria that handle the case when $a$ is even.

\begin{thm}[WF, W\&L]
\label{even}
Suppose $a$ is even, $n$ is odd, and $a^{n}+1$ is the sum of two squares. Then
\begin{itemize}
     \item If $a+1$ is the sum of two squares, then $a^{\delta}+1$ is the sum of two squares for all $\delta|n$, and
     \item If $a+1$ is not the sum of two squares, then there is a unique prime number \\
    $p \equiv 3 \pmod{4},$ such that $p^{r}||a+1$ for some odd $r$, and $n=p$.
\end{itemize}
\end{thm}
\begin{ex}
\end{ex}  
\begin{enumerate}
          \item  For $a \equiv 2 \pmod{4}$, then $a+1$ is not the sum of two squares and so  there is at most one odd exponent $n$ such that  $a^n+1$ is the sum of two squares. For example, with $a=6$, 
             since $a+1 = 7$ is divisible by the unique prime $p=7 \equiv 3 \pmod{4},$  then $n = 7$ is the only possible odd $n$ for which $a^n+1$ is the sum of two squares. Indeed, $6^{7} + 1 = 476^{2} + 231^{2}$.
     \item For $a \equiv 0 \pmod{4}$, there are more options. If we let $a=20$, then
     since $a+1=3\cdot 7$ has two prime factors $\equiv 3 \pmod{4}$ that divide it to an odd power, we conclude that $20^{n} + 1$ is not the sum of two squares for any odd $n$.
On the other hand, for $a=24$, then since $24^{77}+1$ is the sum of two squares (by observation), we must also have that $24^{11}+1,\ 24^{7}+1$, and $24^{1}+1$ are each the sum of two squares.

\end{enumerate}

Additionally we consider a special case when $a$ is a multiple of $4$. 
\begin{thm}[WF]
\label{4x}
Let $a=4x$ where $x\equiv 3 \pmod{4}$ and $x$ is squarefree. If $n$ is odd, then $a^{nx}+1$ is not the sum of two squares. 
\end{thm}

\begin{ex}
\end{ex}
\begin{enumerate}
	\item  Let $a=12 =4\cdot 3$. Then $12^{3n}+1$ is not the sum of two squares for
	any odd $n$. Note that Theorem~\ref{even} implies that since 
$12^{3} + 1$ is not the sum of two squares, then $12^{3n}+1$ is not the sum of two squares for any odd $n$. However, Theorem~\ref{4x} guaranatees, without any computation necessary, that $12^{3} + 1$ is not the sum of two squares.
     \item Let $a=28=4\cdot 7.$ Then $28^{7n}+1$ is not the sum of two squares for any odd $n.$
\end{enumerate}

The factorization tables for $12^{n} + 1$ (\cite{Wagstaff1, Wagstaff2}) imply that there are sixteen exponents $1 \leq n < 293$ for which $12^{n} + 1$ is the sum of two squares, which are all prime except for $n = 1$. For two smallest
composite exponents $n$ for which $12^{n} + 1$ could possibly be the sum
of two squares are $n = 473 = 11 \cdot 43$ and $n = 545 = 5 \cdot 109$;
so far, of those two, we have have confirmed only that $12^{545} + 1$ is the
sum of two squares. 

We now consider the case when $a$ is odd. It's helpful to split this into three
subcases, for $a\equiv 1 \pmod{8}$, 
for $a\equiv 5 \pmod{8}$, 
and  
for $a\equiv 3 \pmod{4}$.

\begin{thm}[WF, W\&L]
\label{1mod8}
Let $a\equiv 1\pmod{8}.$ If $a^n+1$ is the sum of two squares for $n$ odd, then $a^{\delta}+1$ is the sum of two squares for all $\delta|n.$ 
\end{thm}

\begin{ex}
\end{ex}
\begin{enumerate}
     \item Let $a=33.$ Since $33^{119}+1$ is the sum of two squares, then $33^1+1,33^7+1,$ and $33^{17}+1$ must also be sum of two squares. Since $33^3+1$ is not the sum of two squares, we know $33^{3n}+1$ is not the sum of two squares for any odd $n.$
     \item Let $a=41.$ Since $42=2\cdot3\cdot7$ is not the sum of two squares, then $41^1+1$ is not the sum of two squares, and hence 
     $41^n+1$ is not the sum of two squares for any odd $n.$
\end{enumerate}

Note that (as seen in the example with $a=41$) the above theorem implies that if 
$a\equiv 1\pmod{8}$ and 
$a+1$ is not the sum of two squares, then $a^n+1$ is not the sum of two squares for any odd $n.$ The next theorem addresses the case that 
$a \equiv 5 \pmod{8}$. 

\begin{thm}[WF, W\&L]
\label{5}
Let $a\equiv 5 \pmod{8}$. Then,  $a^{n}+1$ is never the sum of two squares for $n$ odd.
\end{thm}
\begin{ex}
\end{ex}
\begin{enumerate}
     \item Since $13\equiv 5 \pmod{8}$, then $13^n+1$ is not the sum of two squares for any odd $n.$
\end{enumerate}

Finally, we consider $a \equiv 3 \pmod{4}$, as covered in three separate results. 
These first two place considerable restrictions on the values of $n$ for which
$a^n+1$ can be a sum of two squares 


\begin{lem}[WF, W\&L]
\label{oddlem1}
Let $a\equiv 3\pmod{4},$
and let $m$ be the smallest integer such that $\frac{a+1}{m}$ is the sum of 
two squares.  If $a^n+1$ is the sum of two squares, then $n\equiv m \pmod{4}.$
\end{lem}

\begin{thm}[WF, W\&L]
\label{3mod4}
Let $a\equiv 3\pmod{4},$
and let $m$ be the smallest integer such that $\frac{a+1}{m}$ is the sum of 
two squares. If $a^{n} + 1$ is a sum of two squares for some odd $n$, then
\begin{itemize}
      \item $\frac{n}{m}$ is a sum of two squares, and
       \item $a^{m}+1$ is the sum of two squares, and
       \item if $\delta \mid \frac{n}{m}$ and
       $\delta$ is the sum of two squares, then $a^{m \delta} + 1$ is the sum of two squares.
       \item Moreover, if $a^{n p^{2}} + 1$ is the sum of two squares for some $p \equiv 3 \pmod{4}$, then $p | a^{n} + 1$.
\end{itemize}
\end{thm}

Theorem~\ref{3mod4} showcases the advantages of having two teams
working independently. When we first shared our results in late July,
the Wake Forest group had only the first two parts of the above
theorem, and the W\&L group had a weaker version of the third part
that was restricted to $m=1$ and to $\delta$ being a prime equivalent
to 1 (mod 4). Two weeks later, both teams had improved their results,
with Wake Forest coming up with both the fourth part and the stronger
version of the third part, as seen here. The proof that resulted from
this collaboration is a nice combination of ideas from both teams.

\begin{ex}
\end{ex}  
\begin{enumerate}
     \item Let $a=11.$ Then $m=3,$ and since $11^{3}+1$ is the sum of two squares, 
     then if $11^n+1$ is the sum of two squares, then $3^{j}||n, \ j$ odd.
     \item Let $a=43.$ Then $m=11,$ and since $43^{11}+1$ is not the sum of two squares, we conclude that 
      $43^{n}+1$ is not the sum of two squares for any odd $n.$
     \item If $a = 4713575$, then $m = 21$. It turns out that $a^{21} + 1$ is the sum of two squares, and so if $a^n+1$ is the sum of two squares, then 
     $21 | n$. Sure enough, $a^{105} + 1$ is the sum of two squares (and has $701$ decimal digits). 
\end{enumerate}


We pause for a moment to remind the reader that Theorem \ref{perfect} states
that if $a$ is not a perfect square, then there exists some odd $n$ such that
$a^n+1$ is not the sum of two squares. We can now extend this theorem
and demonstrate that in fact there will be infinitely many such exponents. 
\begin{itemize} 
\item  If $a$ is even with $a+1$ not the sum of two squares, or if $a \equiv 5 \pmod{8}$,
then Theorems ~\ref{even} and ~\ref{5} tell us that $a^n+1$ fails to be the sum of 
two squares for infinitely many odd $n$ (in fact, for all but at most one
 odd exponent $n$). 
 \item If 
 $a$ is even with $a+1$ the sum of two squares, or if $a \equiv 1 \pmod{8}$,
 then we can use Theorems ~\ref{even} or \ref{1mod8} to state that 
  if $a^\delta+1$ is not the sum of two squares for some odd exponent $\delta$, then
 so also does $a^{\delta N}+1$ fail to be the sum of two squares for all odd
 integers $N$.
\item  Finally, if $ a\equiv 3 \pmod{4}$, we call upon Lemma ~\ref{oddlem1} to state that $a^n+1$ can only be a sum of two squares for $n \equiv m \pmod{4}$.
\end{itemize} 

This next result allows one to state that for certain special values
of $a$, there is an infinite collection of odd values of $n$ for which
$a^{n} + 1$ is the sum of two squares.

\begin{thm}[WF]
\label{px2}
Suppose $n$ is odd, $p\equiv 1 \pmod{4}$ is a prime number and $a=px^{2}$. Then $a^{n}+1$ is the sum of two squares if and only if $a^{np}+1$ is the sum of two squares. 
\end{thm}

The above theorem implies that for those specific values of $a$, then there are either no odd $n$, or an infinite number of odd $n$, for which
$a^{n} + 1$ is the sum of two squares.

\begin{ex}
\end{ex}
\begin{enumerate}
     \item Let $a=17,$ where $p=17$ and $x=1.$ Since $18$ is the sum of two squares, $17^{17^n}+1$ is the sum of two squares for any $n.$
     \item Let $a=117,$ where $p=13$ and $x=3.$ Since $a+1=2\cdot59$ is not the sum of two squares, $117^{13^{n}}+1$ is not the sum of two squares for any $n.$
\end{enumerate}

\begin{rem}
In light of the above theorem, it is natural to ask if there are infinitely many $a \equiv 1 \pmod{8}$
so that $a^{n} + 1$ is the sum of two squares for infinitely many odd $n$. This is indeed the case.
In particular, the main theorem of \cite{Iwaniec} implies that if $x$ is a real number $\geq 17$,
then the number of primes $p \leq x$ with $p \equiv 1 \pmod{8}$ for which $p+1$ is the sum of two squares
is $\geq c \frac{x}{\log(x)^{3/2}}$ for some positive constant $c$.
\end{rem}

We can use the ideas from Theorem~\ref{px2} to 
construct an infinite family of numbers $a$
so that $a^{p} + 1$ is the sum of two squares. This is our
next result.

\begin{thm}[WF]
\label{poly}
If $p \equiv 1 \pmod{4}$ is prime, there is
a degree $4$ polynomial $f(X)$ with integer coefficients
so that $f(X)^{p} + 1 = g(X)^{2} + h(X)^{2}$ for
some $g(X)$ and $h(X)$ with integer coefficients. Moreover,
there is no positive integer $n$ so that $f(n)$ is a perfect square.
\end{thm}

\begin{ex}
\end{ex}   
\begin{enumerate}
\item If $p = 13$, then $f(X) = 13(13X^{2}+3X)^{2}$.
Then $f(n)^{13} + 1$ is the sum of two squares
for every $n \in \N$.
\end{enumerate}

We end with a conjecture about the number of odd $n$
for which $a^{n} + 1$ is the sum of two squares.

\begin{conjec}[WF]
\label{bigconj}
Suppose $a$ is a positive integer and $a \ne c^{k}$ for any positive integer $c$ and $k > 1$. Let $m$ be the smallest positive so that
$\frac{a+1}{m}$ is the sum of two squares.
\begin{itemize}
\item If $m = 1$, then there are infinitely many odd $n$ so that $a^{n} + 1$ is the sum of two squares. 
\item If $a \equiv 3 \pmod{4}$, $a^{m} + 1$ is the sum of two squares, and $m$ is prime,
then there are infinitely many odd $n$ so that
$a^{n} + 1$ is the sum of two squares. (In fact, there should be infinitely many $p \equiv 1 \pmod{4}$ so that $a^{mp} + 1$ is the sum of two squares.)
\item If $a \equiv 3 \pmod{4}$ and $m$ is composite, then there are only finitely many odd $n$ so that
$a^{n} + 1$ is the sum of two squares.
\end{itemize}
\end{conjec}


The main theoretical tools we use in this paper are the theory of cyclotomic polynomials, and in particular, a classification
of which primes divide $\Phi_{n}(a)$ (see Theorem~\ref{cyclothm}). For Theorem~\ref{4x} and Theorem~\ref{px2} also
use the identity $\Phi_{n}(x) = F(x)^{2} - kx^{q} G(x)^{2}$ that arises in Aurifeuillian factorization.

The rest of the paper will proceed as follows. In Section~\ref{back}, we review previous results which we will use. In Section~\ref{niftylemmas}, we prove a few facts that will be used in the remainder of the proofs. In Section~\ref{perfectsquare}, we prove Theorem~\ref{perfect}. In Section ~\ref{evensec}, we prove Theorems ~\ref{even} and ~\ref{4x}. In
 Section~\ref{oddsies}, we prove Theorems 
 ~\ref{1mod8}, ~\ref{5},  and ~\ref{3mod4}, along with Lemma \ref{oddlem1},
  and we include a heuristic supporting Conjecture~\ref{bigconj}. In Section ~\ref{p1mod4}, we prove Theorems ~\ref{px2} and ~\ref{poly}. We conclude with a chart listing all $a \leq 50$ and the first few odd integers $n$ such that $a^{n}+1$ is the sum of two squares, as well as a reference to
one our theorems.



\section{Background}
\label{back}

If $n$ is a positive integer and $p$ is a prime number, we write $p^{r} \| n$ if
$p^{r} | n$ but $p^{r+1} \nmid n$. If $n$ is a positive integer and we write that $n$ is not a sum of two
squares because of the prime $p$, we mean that $p \equiv 3 \pmod{4}$ and there is an odd $r$ so that
$p^{r} \| n$. If $a$ and $m$ are integers with $\gcd(a,m) = 1$, we define
$\ord_{m}(a)$ to be the smallest positive integer $k$ so that
$a^{k} \equiv 1 \pmod{m}$. It is well-known
that $a^{r} \equiv 1 \pmod{m}$ if and only if $\ord_{m}(a) | r$. Fermat's little theorem
states that if $\gcd(a,p) = 1$, then $a^{p-1} \equiv 1 \pmod{p}$; it follows 
that $\ord_{p}(a) | p-1$.

We will make use of the identity (originally
due to Diophantus) that
\[
  (a^{2} + b^{2}) (c^{2} + d^{2})
    = (ac + bd)^{2} + (ad - bc)^{2}.
\]
This applies if the $a, b, c, d \in \Z$, and also
if the $a, b, c$ and $d$ are polynomials.

Let $\Phi_{n}(x)$ denote the $n$th cyclotomic polynomial;
recall that $\Phi_{n}(x)$
is the unique irreducible factor of $x^{n} - 1$ with integer coefficients that does not divide
$x^{k} - 1$ for any proper divisor $k$ of $n$. We have that $\displaystyle \prod_{d | n} \Phi_{d}(x) = x^{n} - 1$
and from this it follows that when $n$ is odd, 
\[
  x^{n} + 1 = \frac{x^{2n} - 1}{x^{n} - 1} = \prod_{\substack{d | 2n \\ d \nmid n}} \Phi_{d}(x) = \prod_{\delta | n} \Phi_{2 \delta}(x).
\]
We will make use of the facts that for $n$ odd,
$\Phi_{2n}(x) = \Phi_{n}(-x)$ and also that
if $n = p^{k}$ is prime, then $\Phi_{p^{k}}(1)
= \lim_{x \to 1} \frac{x^{p^{k}} - 1}{x^{p^{k-1}} - 1} = p$. 

The following theorem classifies prime divisors of $\Phi_{n}(a)$.

\begin{thm}
\label{cyclothm}
Assume that $a \geq 2$ and $n \geq 2$.
\begin{itemize}
\item If $p$ is a prime and $p \nmid n$, then
$p | \Phi_{n}(a)$ if and only if $n= \ord_{p}(a)$. 
\item If $p$ is a prime and $p | n$, then
$p | \Phi_{n}(a)$ if and only if $n = \ord_{p}(a)\cdot p^{k}$. 
In this case, when $n \geq 3$, then 
$p^{2} \nmid \Phi_{n}(a)$.
\end{itemize}
\end{thm}

This theorem arises in connection with Zsigmondy's work showing that
for any $a, n \geq 2$ there is a prime $p$ for which $\ord_{p}(a) = n$
unless $n = 2$ and $a+1$ is a power of $2$. A proof of
Theorem~\ref{cyclothm} is given in \cite{Roitman} (see Proposition~2),
but Roitman indicates that this theorem was stated and proved earlier
by L\"uneberg (see Satz~1 of \cite{Luneburg}).


We will also make use of certain identities for cyclotomic polynomials
that arise in Aurifeuillian factorization.  If $k$ is a squarefree
positive integer, let $d(k)$ be the discriminant of $\Q(\sqrt{k})$,
that is,
\[
  d(k) = \begin{cases}
    k & \text{ if } k \equiv 1 \pmod{4}\\
    4k & \text{ if } k \equiv 2, 3 \pmod{4}.
  \end{cases}
\]
Suppose that $n\equiv 2 \pmod{4}$, 
and $d(k) \nmid n$ but $d(k) | 2n$. Write the
prime factorization of $n$ as
$\displaystyle n = 2 \prod_{i=1}^{k} p_{i}^{e_{i}}$ and define
$\displaystyle q = \prod_{i=1}^{k} p_{i}^{e_{i}-1}$. Then Theorem~2.1 of \cite{Stevenhagen}
states that
\[
  \Phi_{n}(x) = F(x)^{2} - k x^{q} G(x)^{2}
\]
for some polynomials $F(x), G(x) \in \Z[x]$. 
In the case that $x = -kv^{2}$ for some integer $v$ we get that
\[
  \Phi_{n}(-kv^{2}) = F(-kv^{2})^{2} + \left(k^{\frac{q+1}{2}} v^{q} G(-kv^{2})\right)^{2}
\]
is the sum of two squares. In the case
that $x = kv^{2}$ for some integer $v$, we get a factorization
\begin{align*}
  \Phi_{n}(kv^{2}) &= F(kv^{2})^{2} - k (kv^{2})^{q} G(kv^{2})^{2}\\
  &= \left(F(kv^{2}) + k^{\frac{q+1}{2}} v^{q} G(kv^{2})\right)
  \left(F(kv^{2}) - k^{\frac{q+1}{2}} v^{q} G(kv^{2})\right).
\end{align*}
Theorem 2.7 of \cite{Stevenhagen} states that these two factors
are relatively prime. 

We will also require some basic facts about quadratic residues.
If $p$ is an odd prime, we define $\legen{a}{p}$ to be $1$ if
$\gcd(a,p) = 1$ and there is some $x \in \Z$ so that $x^{2} \equiv a \pmod{p}$.
We define $\legen{a}{p}$ to be $-1$ if $\gcd(a,p) = 1$ and there is no such $x$,
and we set $\legen{a}{p} = 0$ if $p | a$. Euler's criterion gives
the congruence $\legen{a}{p} \equiv a^{\frac{p-1}{2}} \pmod{p}$.
We will also use the law of quadratic
reciprocity, which states that if $p$ and $q$ are distinct odd primes,
then $\legen{p}{q} \legen{q}{p} = (-1)^{\frac{p-1}{2} \cdot \frac{q-1}{2}}$.

The definition of the quadratic residue symbol can be extended. If $n$
is an odd integer with prime factorization $\displaystyle n = \prod_{i=1}^{k} p_{i}^{e_{i}}$,
define $\displaystyle \legen{a}{n} = \prod_{i=1}^{k} \legen{a}{p}^{e_{i}}$.


\section{General Results}
\label{niftylemmas}
The following general lemmas pertain primarily to how the divisors
of $n$ affect the divisors of $a^{n}+1$, and are used in rest of the sections of the paper. 
\begin{lem}
\label{general1}
Let $b,n \in \Z$, and $n$ be odd and suppose $b|x+1$. Then $b|(x^{n-1}-x^{n-2}+x^{n-3}-\cdots+1)$ if and only if $b|n$. 

\end{lem}
\begin{proof}

Let $b|x+1.$ Then $x+1 \equiv 0 \pmod{b},$ so  $x \equiv -1 \pmod{b}.$ Then,
    \begin{align*}
        x^{n-1}-x^{n-2}+x^{n-3}-&\cdots-x+1 \\
        &\equiv (-1)^{n-1}-(-1)^{n-2}+(-1)^{n-3}-\cdots-(-1)+1 \pmod{b} \\
        &\equiv    1+1+1+\cdots+1+1 \pmod{b} \\
        & \equiv    n \pmod{b}.  
        \end{align*}

Therefore $b|x^{n-1}-x^{n-2}+x^{n-3}-\cdots-x+1$ if and only if $ n \equiv 0 \pmod{b},$ or equivalently, $b|n$.
  
\end{proof}

We obtain the following corollary as a result of the above lemma. 

\begin{cor}
\label{general2}
Suppose $\delta |n$ and $x^{\delta}+1$ is not the sum of two squares because of some prime $p.$ If $p\nmid n,$ then $x^{n}+1$ is not the sum of two squares.
\end{cor}

\begin{proof}
Consider $$x^{n}+1=(x^{\delta}+1)(x^{n-\delta}-x^{n-2\delta}+x^{n-3\delta}-\cdots-x^{\delta}+1).$$ Since $x^{\delta}+1$ is not the sum of two squares because of $p$, we have $p\equiv 3 \pmod{4},\ r$ odd and $p^{r}||x^{\delta}+1$. Then $p \nmid n$ implies $p\nmid x^{n-\delta}-x^{n-2\delta}+x^{n-3\delta}-\cdots-x^{\delta}+1$ by Lemma ~\ref{general1}, and thus $p^{r}||x^{n}+1$ and implying that $x^{n}+1$ is not the sum of two squares.  
\end{proof}

\begin{lem}
\label{general3}
Let $p$ be a prime such that $p^{e}||a^{m}+1$ for some $e \in \mathbb{N}$, and let $n=mcp^{k}$ with $\gcd(c, p) = 1$ and $k\geq 0.$ Then $p^{e+k}||a^{n}+1$.
\end{lem}
\begin{proof}
Using notation from the statement of the theorem, we can write:
\[
a^n+1 = \left(a^{m}+1\right) \cdot \left(\frac{a^{n}+1}{a^m+1}\right).
\]
Then, recalling how $a^m+1$ factors into cyclotomics, 
we let $d$ be the smallest divisor of $m$ such that $p | \Phi_{2d}(a)$. 
Thanks to Theorem~\ref{cyclothm}, we know that 
$p|| \Phi_{2dp}(a)$, $p|| \Phi_{2dp^2}(a)$, and so on, yet
$p$ does not divide into any other cyclotomic expressions not of that form.  
Now, choose $i$ as large as possible such that $2dp^i | m$. 
Then, by our definition of $n$,
we know that everything in the set $\{dp^{i+1}, dp^{i+2}, \dots, dp^{i+k}\}$
divides into $n$ yet none of them divide into $m$, and we also know from 
Theorem~\ref{cyclothm} (as mentioned above) that each
of the $k$ expressions 
$\Phi_{2dp^{i+1}}(a), \Phi_{2dp^{i+2}}(a), \dots, \Phi_{2dp^{i+k}}(a)$ contains
exactly one copy of the prime $p$ and that no other cyclotomic divisors of
$\frac{a^n+1}{a^m+1}$ contain this prime $p$. Hence, 
since $p^{e}||a^{m}+1$, 
then $p^{e+k}||a^{n}+1$.
\end{proof}


\section{Proof of Theorem ~\ref{perfect}}
\label{perfectsquare}

This section will begin with the proofs of several lemmas, from which the proof of Theorem ~\ref{perfect} is constructed.

\begin{lem}
\label{step1}
Suppose there exists a prime $p \equiv 3 \pmod{4}$ such that $\legen{a}{p} = -1$. Then either $a^{\frac{p-1}{2}}+1$ or $a^{\frac{p(p-1)}{2}}+1$ is not a sum of two squares.
\end{lem}
\begin{proof}
If $a^{\frac{p-1}{2}}+1$ is not a sum of two squares, then we are done. Suppose $a^{\frac{p-1}{2}}+1$ is a sum of two squares. By Euler's criterion,
we have that $a^{\frac{p-1}{2}} \equiv -1 \pmod{p}$, and
it follows therefore that for some $k \in \mathbb{N}$, $p^{2k} \parallel a^{\frac{p-1}{2}}+1$. By Lemma~\ref{general3}, letting $m = \frac{p-1}{2}$ and $n = \frac{p(p-1)}{2}$, we know that $p^{2k+1} \parallel a^{\frac{p(p-1)}{2}}+1$. Thus, by Fermat's two squares theorem, $a^{\frac{p(p-1)}{2}}+1$ is not the sum of two squares.
\end{proof}
As an example, we examine $148^n + 1$. We can conclude from the prime factorization of $148^n + 1$ that $148^n + 1$ is a sum of two squares for all odd $n < 9$. Note that $9 = \frac{19-1}{2}$ and that $19$ is the smallest prime $p \equiv 3 \pmod{4}$ for which the Legendre symbol $\legen{148}{p} = -1$. Calculation and Fermat's two square theorem reveal that $148^{\frac{19-1}{2}} + 1 = 148^{9} +1$ is not a sum of two squares.

Before we continue on to the next lemma, we need to define the following function for an integer $a$:
$$\chi_{a}(n) = \begin{cases} \legen{a}{n} &\textrm{ if } n \textrm{ is odd, and}\\ 0 &\textrm{ if } n \textrm{ is even.} \end{cases}$$

\begin{lem}
\label{step3}
If $p$ is an odd prime, define
$$p^{*} = \begin{cases} p &\textrm{ if } p \equiv 1 \pmod{4}\\ -p &\textrm{ if } p \equiv 3 \pmod{4}. \end{cases}$$
If $q$ is an odd prime, then $\chi_{p^*}(q) = \legen{q}{p}$. Additionally, there is some $\delta \in \{1, -1, 2, -2\}$ such that for each $n$ with $\gcd(a, n) = 1$, $$\chi_{a}(n) = \chi_{\delta}(n) \prod_{p^r \parallel a \text{ with } r \text { odd}} \chi_{p^*}(n).$$
\end{lem}

\begin{proof}
Because $p$ and $q$ are both prime,
$$\chi_{p^*}(q) = \legen{p^*}{q} = \begin{cases} \legen{p}{q} &\text{ if } p \equiv 1 \pmod{4}\\ \\ \legen{p}{q} \legen{-1}{q} &\text{ if } p \equiv 3 \pmod{4}. \end{cases}$$
If $p \equiv 1 \pmod{4}$ or $q \equiv 1 \pmod{4}$, then $\legen{p}{q} = \legen{q}{p}$ by the law of quadratic reciprocity.\\
If $p \equiv 3 \pmod{4}$ and $q \equiv 3 \pmod{4}$, then by the law of quadratic reciprocity  $\legen{p}{q} = -\legen{q}{p}$. In addition,  $\legen{-1}{q} = (-1)^{\frac{q-1}{2}} = -1$. Then $\legen{p^*}{q} = \legen{p}{q} \legen{-1}{q} = -\legen{q}{p} (-1) = \legen{q}{p}$. Thus $\chi_{p^*}(q) = \legen{q}{p}$.\\
If $n$ is even, then $\chi_{a}(n) = 0$.
Looking at odd $n$, we note $\displaystyle \prod_{\substack{p^r \parallel a\\q^l \parallel n}} \legen{p}{q} = (-1)^d \prod_{\substack{p^r \parallel a\\q^l \parallel n}} \legen{q}{p}$, where the exponent $d$ depends on $n$ and $a$. From the law of quadratic reciprocity, we know that $\legen{p}{q} = - \legen{q}{p}$ only if both $p$ and $q$ are congruent to $3 \pmod{4}$. Then $d$ is the product of the number of primes $p \equiv 3 \pmod{4}$ which, to an odd power, exactly divide $a$, and the number of primes $q \equiv 3 \pmod{4}$ which, to an odd power, exactly divide $n$. Thus, if $2^b \parallel a$, then $d$ is even if $\dfrac{a}{2^b} \equiv 1 \pmod{4}$ or $n \equiv 1 \pmod{4}$ and $d$ is odd if both $\dfrac{a}{2^b}$ and $n$ are congruent to $3 \pmod{4}$. This is equivalent to writing $d = \dfrac{n-1}{2}\dfrac{\dfrac{a}{2^b}-1}{2}$. Then
\begin{eqnarray*}
\chi_{a}(n) &= \legen{a}{n} &= \left( \dfrac{2^{b}}{n} \right) \prod_{\substack{p^r \parallel a\\ q^l \parallel n}}\legen{p}{q}^{rl}\\ 
& & = \left( \dfrac{2^{b}}{n} \right) (-1)^{\frac{n-1}{2}\frac{\frac{a}{2^b}-1}{2}} \prod_{\substack{p^r \parallel a\\ q^l \parallel n}}\legen{q}{p}^{rl}\\
& & =\left( \dfrac{(-1)^{\frac{\frac{a}{2^b}-1}{2}}2^{b}}{n} \right) \prod_{\substack{p^r \parallel a\\ q^l \parallel n}}\legen{q}{p}^{rl}\\
& & = \chi_{\delta}(n)  \prod_{p^r \parallel a} \chi_{p^*}(n),
\end{eqnarray*}
where $\delta \in \{ -1, 1, -2, 2 \}$.
\end{proof}

In the following lemma, we examine the case where $a$ is twice a square.
\begin{lem}
\label{step4}
If $a$ is twice a square, then there is a prime $q \equiv 3 \pmod{8}$ so that $\chi_{\delta}(q) = -1$, and hence $\chi_{a}(q) = -1$.
\end{lem}
\begin{proof}
Since there are infinitely many primes congruent to $3 \pmod{8}$, it is always possible to pick such a prime $q$ which does not divide $a$ and has the property $\chi_{2}(q) = -1$. Because $q \nmid a$ and $a$ is twice a square, by Lemma~\ref{step3} we can write
$$\chi_{a}(q) = \chi_{2}(q) \prod_{p^{2k} \parallel a}\chi_{p^*}(q) = \chi_{2}(q) (1)=-1.$$
\end{proof}

To be complete, we must now examine the case where $a$ is neither a square nor twice a square.
\begin{lem}
\label{step5}
If $a$ is not a square and not twice a square, then there is a prime $q \equiv 3 \pmod{4}$ so that $\chi_{a}(q) = -1$.
\end{lem}
\begin{proof}
Suppose $\delta = -1$. Then for all $r_i$ such that $r_i^{2k + 1} \parallel a$ for some $k \in \mathbb{Z}$, there is a set of congruences $q \equiv 1 \pmod{r_i}$ and $q \equiv 3 \pmod{4},$
which can be solved for a prime $q$ that has the property $\chi_{r_i}(q) = 1$ for all $r_i$ and $\chi_{\delta}(q) = -1$. Then
$$\chi_{a}(q) = \chi_{\delta}(q) \prod_{r_i^{2k + 1} \parallel a}\chi_{r_i^*}(q) = (-1)(1)=-1.$$
Alternatively, suppose $\delta \neq -1$. Then there is a congruence class$\pmod{8}$ so that if $q \equiv \text{c} \pmod{8}$, then $q \equiv 3 \pmod{4}$ and $\chi_{\delta}(q) = 1$. Then for all $r_i$ such that $r_i^{2k + 1} \parallel a$, there is a set of congruences
\begin{eqnarray*}
q &\equiv &\text{quadratic non-residue} \pmod{r_j} \text{ for some set } r_j\\
q &\equiv &1 \pmod{r_i} \text{ for all } i \neq j\\
q &\equiv &\text{c} \pmod{8}
\end{eqnarray*}
which can be solved for a prime $q$ that has the property $\chi_{r_i}(q) = 1$ for all $i \neq j$ and $\chi_{r_j}(q) = -1$. Then
$$\chi_{a}(q) = \chi_{\delta}(q) \prod_{r_i^{2k + 1} \parallel a}\chi_{r_i^*}(q) = (1)(-1)=-1.$$
Thus it is always possible to find some prime $q \equiv 3 \pmod{4}$ such that $\chi_{a}(q) = -1$.
\end{proof}

Using these lemmas, we can now prove Theorem ~\ref{perfect}.
\begin{proof}[Proof of Theorem ~\ref{perfect}]
Note that any number that is not a square is either twice a square or not twice a square. Lemmas ~\ref{step4} and ~\ref{step5} show that for any number that is not a square, it is possible to pick a prime $q \equiv 3 \pmod{4}$ such that $\chi_{a}(q) = -1$. Hence $\legen{a}{q} = -1$ and so by Lemma ~\ref{step1}, either $a^{\frac{p-1}{2}} + 1$ or $a^{\frac{p(p-1)}{2}} + 1$ is not a sum of two squares and so there is at least one value of $n$ for which $a^n + 1$ is not a sum of two squares.
\end{proof}


\section{Even}
\label{evensec}


Now we consider the case when $a$ is even. We prove Theorems~\ref{even}
and \ref{4x}.


\begin{proof}[Proof of Theorem ~\ref{even}]
Suppose that $a^{n} + 1$ is the sum of two squares. If 
$a^\delta + 1$ is also the sum of two squares for every divisor $\delta$ of $n$, then we are done. If not, then let
$\delta$ be the largest divisor of $n$ so that $a^{\delta} + 1$ is not
the sum of two squares. Thus, $\delta < n$ and so there is a prime $p$
that divides $n/\delta$. By assumption, we have that $a^{\delta p} + 1$ is
the sum of two squares and
\[
  a^{\delta p} + 1 = (a^{\delta} + 1) (a^{\delta (p-1)} - a^{\delta (p-2)} + \cdots + 1).
\]
Lemma ~\ref{general1} implies that $\gcd\left(a^{\delta} + 1, \frac{a^{\delta p} + 1}{a^{\delta} + 1}\right)$ divides $p$. Since $a^{\delta} + 1$ is not the sum of two
squares, the gcd cannot be $1$ and so it must be $p$. Moreover,
\[
  \frac{a^{\delta p} + 1}{p^{2}}
  = \left(\frac{a^{\delta} + 1}{p}\right) \left(\frac{a^{\delta p} + 1}{p(a^{\delta} + 1)}\right)
\]
is a sum of two squares and the product of two relatively prime integers.
Thus, $\frac{a^{\delta} + 1}{p}$ is the sum of two squares. It follows
that $p \equiv 3 \pmod{4}$ and since $a^{\delta} + 1$ is odd, we get
\[
  a^{\delta} + 1 = p \times \text{sum of two squares} \equiv 3 \pmod{4}.
\]
However, since $a$ is even, we must have that $\delta = 1$ and the
previous equation implies that $p$ is the unique prime $\equiv 3 \pmod{4}$
that divides $a+1$ to an odd power.
\end{proof}



Let us consider a special case of even $a$, where $a$ is a multiple of $4$.

\begin{proof}[Proof of Theorem ~\ref{4x}]

First, we show that $a^{x} + 1$ is not the sum of two squares. We have that
\[
  a^{x} + 1 = \prod_{\substack{d | 2x \\ d \nmid x}} \Phi_{d}(a).
\]
We apply Theorem 2.1 of \cite{Stevenhagen} to $\Phi_{2x}(y) \in \Z[y]$.
We set $n = 2x$, $k = x$, $d(k) = 4x$. Then $d(k) \nmid n$ but $d(k) \mid 2n$.
We have that
\[
  \Phi_{2x}(y) = F(y)^{2} - x y G(y)^{2}.
\]
Assume without loss of generality that the leading coefficient of $F(y)$ is positive. Note that since $\Phi_{2x}(y)$ has even degree, the degree of
$F(y)$ is larger than that of $G(y)$.

Replacing $y$ with $xy^{2}$ we get
\[
  \Phi_{2x}(xy^{2}) = F(xy^{2})^{2} - x (xy^{2}) G(xy^{2})
  = \left(F(xy^{2}) + xy G(xy^{2})\right) \left(F(xy^{2}) - xy G(xy^{2})\right).
\]
Let $f(y)$ and $g(y)$ be the first and second factors above,
respectively.  We have $\Phi_{2x}(a) = \Phi_{2x}(4x) = f(2)
g(2)$. From Theorem 2.7 of \cite{Stevenhagen} we know that
$\gcd(f(2),g(2)) = 1$.  We claim that
$f(2) \equiv g(2) \equiv 3 \pmod{4}$. This will follow
if we show that the constant coefficients of $f(y)$ and $g(y)$ are both
$1$, and the linear coefficients of $f(y)$ and $g(y)$ are both odd.

We have that $f(y) = a_{0} + a_{1} y + a_{2} y^{2} + \cdots$
and $g(y) = a_{0} - a_{1} y + a_{2} y^{2} + \cdots$. Since the constant
coefficient of $\Phi_{2x}(y)$ is $1$, we have that $a_{0}^{2} = 1$ and
so $a_{0} = \pm 1$. If $a_{0} = -1$, then since the leading coefficient
of $F(y)$ is positive, $f(y)$ and $g(y)$ have positive leading coefficients.
However, then $\lim_{y \to \infty} f(y) = \lim_{y \to \infty} g(y) = \infty$
but $f(0) = g(0) = -1$. This implies that $f(y)$ and $g(y)$ both have
a positive real root, but $f(y) g(y) = \Phi_{2x}(xy^{2})$ has no real
roots. This is a contradiction and so $a_{0} = 1$.

It is well-known that if $n > 1$, the coefficient of $y$ in $\Phi_{n}(y)$
is $-\mu(n)$ (see for example, the last equation on page 107 of
\cite{Lehmer}). 
Multiplying $f(y)$ and $g(y)$, we get
\[
  \Phi_{2x}(xy^{2}) = 1 - \mu(2x) xy^{2} + \cdots = a_{0}^{2} + (2a_{0} a_{2} - a_{1}^{2}) y^{2} + \cdots.
\]
We have that $\mu(2x) = \pm 1$ is odd and $-\mu(2x) = 2a_{0} a_{2} - a_{1}^{2}$.
Thus, $a_{1}^{2} \equiv \mu(2x) \pmod{2}$ and so $a_{1}$ is odd.
Thus, $f(2) \equiv a_{0} + 2a_{1} \equiv 1 + 2 \equiv 3 \pmod{4}$
and likewise $g(2) \equiv a_{0} - 2a_{1} \equiv 1 - 2 \equiv 3 \pmod{4}$.

Thus, there is a prime $p \equiv 3 \pmod{4}$ and an odd $j$ so that 
$p^{j} \| f(2)$ and a prime $q \equiv 3 \pmod{4}$ and an odd
$k$ so that $q^{k} \| g(2)$. Since $\gcd(f(2),g(2)) = 1$,
we have $p \ne q$. 

We claim that at most one of $p$ or $q$ divides $x$. Suppose
to the contrary that $p | x$ and $q | x$. Since $p | \Phi_{2x}(a)$,
Theorem ~\ref{cyclothm} implies that $2x = p\cdot\ord_{p}(a)$ and
since $q | \Phi_{2x}(a)$, we get that $2x = q \cdot \ord_{q}(a)$.
This implies that $\ord_{p}(a) = \frac{2x}{p}$ is a multiple of $q$
and $\ord_{q}(a) = \frac{2x}{q}$ is a multiple of $p$. This is a contradiction,
because either $p < q$ (in which case $q \leq \ord_{p}(a) \leq p-1$)
or $q < p$ (in which case $p \leq \ord_{q}(a) \leq q-1$).

Thus, at most one of $p$ or $q$ divides $x$. Assume without loss of generality
that $p \nmid x$. Then we have that $p^{j} \| \Phi_{2x}(a)$ and
Theorem ~\ref{cyclothm} gives that $\ord_{p}(a) = 2x$. This implies that
$p \nmid \Phi_{2 \delta}(a)$ for $\delta | x$ with $\delta \ne x$.
As a consequence, $p^{j} \| a^{x} + 1$ and so $a^{x} + 1$ is not the sum of
two squares.

Now, let $A = a^{x}$. Then $A+1$ is not the sum of two squares, and
$A+1 \equiv 1 \pmod{4}$. Thus, there are at least two primes $\equiv 3 \pmod{4}$
that divides $A+1$ to an odd power, and Theorem \ref{even} implies that $A^{n} + 1$ is never the
sum of two squares for $n$ odd.

\end{proof}


\section{odd}\label{oddsies}

This section contains proofs of Theorems 
 ~\ref{1mod8}, ~\ref{5},  and ~\ref{3mod4}, along with Lemma \ref{oddlem1}, 
which pertain to when $a^n+1$ can be written as a sum of two squares when $a$ is an odd integer. 
In this section, we define $m$ to be the least positive integer such that $\frac{a+1}{m}$ is the sum of two squares.


We begin with $a\equiv 1 \pmod{4}$. We  prove Theorem ~\ref{1mod8} which handles the case $a\equiv 1\pmod{8},$ and Theorem ~\ref{5} which handles $a\equiv 5 \pmod{8}.$


\begin{proof} [Proof of Theorem ~\ref{1mod8}]
Let $a\equiv 1\pmod{8}.$ Then $a^n+1\equiv 2\pmod{8}$ for all $n$, so
$\frac{a^n+1}{2}\equiv 1\pmod{4}.$ Suppose $a^n+1$ is the sum of two squares,
and assume by contradiction that $\delta$ is the largest divisor of $n$ such
that $a^\delta+1$ is not the sum of two squares. 
Since $\frac{a^\delta+1}{2} \equiv 1 \pmod{4}$, 
then there exist distinct primes $q_1\equiv q_2 \equiv 3\pmod{4}$ such that 
$q_1^{j_1}||a^\delta+1$ and $q_2^{j_2}||a^\delta+1, \ j_1, j_2$ odd.
\par
We know from Lemma ~\ref{general3} that since $a^n+1$ is the sum of 
two squares, $q_{1}^{l_{1}} \parallel n$ and
$q_{2}^{l_{2}} \parallel n$ for some odd
$l_{1}$ and $l_{2}$. Without loss of generality, suppose 
$q_1>q_2,$ and consider:
\begin{align*}
    a^{\delta q_1}+1=\big(a^\delta+1\big) \prod_{\substack{\delta_x|\delta q_1 \\ 
    \delta_x \nmid \delta}} \Phi_{2\delta_x}(a).
\end{align*}
Since $q_1>q_2,$ we know $q_1\nmid \ord_{q_2}(a),$ and Theorem~\ref{cyclothm} implies that $q_2\nmid \frac{a^{\delta 
q_1}+1}{a^\delta+1}.$ Then $q_2^{j_2}||a^{\delta q_1}+1,$ so $a^{\delta q_1}+1$ is 
not the sum of two squares. This is a contradiction because $\delta q_1>\delta$ and 
$\delta q_1|n.$ Thus $a^\delta+1$ is the sum of two squares for all $\delta|n.$
\end{proof}


 \begin{proof}[Proof of Theorem ~\ref{5}]
    
    Suppose $a\equiv 5\pmod{8}$ and $n$ is odd. Then: 
        \begin{eqnarray*}
        a^{n}+1&=&a^{2k+1}+1\\
        &\equiv &5^{2k}\cdot5+1 \pmod{8}\\
        &\equiv &6 \pmod{8}.
        \end{eqnarray*}
This implies that $\frac{a^{n}+1} {2} \equiv 3 \pmod{4},$ so by
Fermat's two squares theorem we know that $a^n+1$ is never the sum of two squares when $n$ is odd. 
    \end{proof}

Next, the following  lemmas will be useful in forming contradictions in the proof of Theorem ~\ref{3mod4} because of the restrictions they place on $n$ in order for $a^n+1$ to be the sum of two squares, where $a\equiv 3\pmod{4}$ and $n$ odd.


We begin with two lemmas that cover the modulus of permissible exponents $n$ when 
$a \equiv 3 \pmod{4}$.


\begin{lem}\label{4jplus1}
For $a = 4\cdot 2^i \cdot (4j+1)-1$ with $i, j \geq 0$,
 then $a^n+1$ can only be written as the sum of two squares (for $n$ odd) if $n\equiv 1$ mod $4$. \end{lem}

Note that this covers values of $a$ such as
$a=3$, 7, 15, 19, 31, and 35. This  explains why $35^9+1$ is a sum of two 
squares but $35^3+1$ is not. 

\begin{proof}
Let us argue by contradiction.
Suppose $ n\equiv 3$ mod $4$. Write $n=4k+3$, and 
note that $a \equiv 4\cdot 2^i - 1$ mod $16\cdot 2^i$.
Then, 
making liberal use of the binomial theorem on $a^3 \equiv (4\cdot2^i-1)^3$ and 
$a^4 \equiv (4\cdot 2^i-1)^4$,
we have:
\begin{align*}
    a^n+1 &= a^{4k+3} +1 \\
          &= (a^3)\cdot(a^4)^k+1 \\
          &\equiv \Big(\dots + 3\cdot (4 \cdot 2^i) - 1\Big)\cdot \Big(\dots - 4\cdot (4 \cdot 2^i) + 1\Big)^k+1 \ \ \ \ \mbox{mod 16}\cdot 2^i\\
          &\equiv \Big(3\cdot 4 \cdot 2^i - 1\Big)\cdot (1)^k+1  \ \ \ \   \mbox{mod } 16\cdot 2^i\\
          &\equiv 12\cdot 2^i \ \ \ \ \mbox{mod } 16\cdot 2^i.
\end{align*}

This implies that  $\frac{a^n+1}{4\cdot 2^i}$ is equivalent to $3$ mod $4$. Then there must be at least one prime equivalent to 3 mod 4 that appears in the factorization of $\frac{a^n+1}{4\cdot 2^i}$ an odd number of times. This implies the same for $a^n+1$ and thus by Fermat, $a^n+1$ is not the sum of two squares. This is a contradiction to our assumption and thus $n$ cannot be equivalent to $3$ mod $4$.
\end{proof}

\begin{lem}\label{4jplus3}
For $a = 4\cdot 2^i \cdot (4j+3)-1$ with $i, j \geq 0$, then $a^n+1$ can only be written as the sum of two squares (for $n$ odd) if $n\equiv 3$ mod $4$. 
\end{lem}

Note that this covers values of $a$ such as
$a=11$, 23, 27, 43, and so on, including 
191 which gives us two values $n=3$ and $n=15$ such that $191^n+1$ is the sum of two squares. Both $3$ and $15$, of course,
are equivalent to 3 mod 4. 

\begin{proof}
Keeping in mind that $a \equiv -1$ mod 4,
we have:
\begin{align*}
    a^n+1 &= (a+1)\cdot(a^{n-1}-a^{n-2}+ \cdots +1)\\
    &=4\cdot 2^i \cdot(4j+3)\cdot(a^{n-1}-a^{n-2}+ \cdots +1)
\end{align*}
Since $a \equiv -1$ mod 4, then that last expression, 
$(a^{n-1}-a^{n-2}+ \cdots +1)$, is equivalent to $n$ mod 4. The only hope, then,
for $a^n+1$ to be the sum of two squares is for $n$ to be 3 mod 4, as then
$\frac{a^n+1}{4\cdot 2^i}$ will be the product of two expressions both equivalent to
3 mod 4, resulting in 
$\frac{a^n+1}{4\cdot 2^i}$ being equivalent to 1 mod 4.
\end{proof}

The last two lemmas allow us to now prove one of our earlier lemmas:

\begin{proof}[Proof of Lemma ~\ref{oddlem1}]
For $a \equiv 3 \pmod{4}$, we can write $a=4K-1$, where $K$ can be split into
an even part (which we write as $2^i$) and an odd part (which we write as either 
$4j+1$ or $4j+3$). In the first case, $a+1$ 
equals $4\cdot 2^i \cdot (4j+1)$ and
since $m$ is the smallest integer such that $\frac{a+1}m$ is the sum of two squares, 
then $m$ must be equivalent to $1 \pmod{4}$, and by 
Lemma ~\ref{4jplus1} we have 
$n \equiv 1 \pmod{4}$ in this case, and so 
$n \equiv m \pmod{4}$. A similar argument applies to the second case. 
\end{proof}


This lemma places further restrictions on $n$. Recall that $m$ is
the smallest positive integer so that $\frac{a+1}{m}$ is the sum of
two squares.

\begin{lem}
\label{oddlem2}
Let $a\equiv 3\pmod{4}.$ If $a^n+1$ is the sum of two squares, then for all primes $p\equiv 3 \pmod{4}$ such that $p^{e}||a+1,$  $e$ odd, we have $p^{k}||n,$ $k$  odd. In particular, if $a^n+1$ is the sum of two squares, then $m|n.$
\end{lem}
\begin{proof}
Let $a^n+1$ be the sum of two squares and suppose $p^{e}||a+1, e$ odd, and $p\equiv 3\pmod{4}.$ 
Select $k$ such that $p^{k} || n$. Then, 
Lemma ~\ref{general3} implies that $p^{e+k} || a^{n}+1$.
Since $a^n+1$ is the sum of two squares, we know $e+k$ is even, which makes $k$ odd.
It follows that since $m=\prod p$
for $p$ such primes of this type, then if $a^n+1$ is the sum of two squares, then $m|n.$ 
\end{proof}

We will now prove Theorem ~\ref{3mod4}, which applies to all $a\equiv 3\pmod{4}.$

\begin{proof}[Proof of Theorem ~\ref{3mod4}]
  First we will prove that $\frac{n}{m}$ is the sum of two
  squares. Suppose that $a^{n} + 1$ is the sum of two squares and
  recall that by Lemma~\ref{oddlem2} that $m | n$. Assume by
  contradiction that $\frac{n}{m}$ is not the sum of two squares. Then
  let $q$ be the greatest prime such that $q\equiv 3\pmod{4}$ and
  $q^j\parallel \frac{n}{m}, \ j $ odd. If $q|m$, then
  Lemma~\ref{general3} implies that an even power of $q$ divides
  $a^{m} + 1$, and so if an odd power of $q$ divides $a^{n} + 1$, then
  $q^r\parallel n,\ r$ odd. But $m$ is squarefree, so $q \parallel m.$
  Then $q^{r-1}\parallel \frac{n}{m}, \ r-1$ even, which is a
  contradiction. Therefore we can assume $q\nmid m,$ so
  $q^j\parallel n.$

We know that $\Phi_{2q^{j}}(a)$ divides $a^{n} + 1$. We have that
$\Phi_{2q^j}(a)\equiv\Phi_{2q^j}(-1)\equiv \Phi_{q^j}(1)\equiv q
\equiv 3 \pmod{4}.$
This implies that there exists a prime $p\equiv 3\pmod{4}$ such that
$p^k\parallel \Phi_{2q^j}(a), \ k$ odd. We can consider two cases:
when $p\neq q,$ and when $p=q.$

Suppose $p\neq q.$ Then $p\nmid q^j,$ so $\ord_p(a)=2q^j,$ which implies $p>q.$  Since $a^n+1$ is the sum of two squares, Lemma ~\ref{general3} implies that $p^l\parallel n, \ l$ odd. Since $\ord_p(a)>2,$ $p\nmid a+1,$ so $p\nmid m.$ Then $p$ is a prime congruent to $3\pmod{4}$ that divides $\frac{n}{m}$ to an  odd power, and $p>q$, which is a contradiction because we assume $q$ is the largest such prime.

Now suppose $p=q$. Since $p | \Phi_{2p^{j}}(a)$ it follows that
$a^{p^{j}} + 1 \equiv 0 \pmod{p}$. Repeatedly applying Fermat's little theorem,
that $a^{p} \equiv a \pmod{p}$, we find that $p | a+1$.
Since $p \nmid m$, $p^{k} \| a+1$, $k$ even. Then Lemma ~\ref{general3} implies that $p^{k+j}\parallel a^{n}+1,$ where $k+j$ is odd, which is a contradiction. Thus if $a^n+1$ is the sum of two squares, then $\frac{n}{m}$ is also the sum of two squares.  

Next we'll prove that $a^m+1$ is the sum of two squares. Suppose
$a^n+1$ is the sum of two squares, where $n=ms,$ and assume by
contradiction that $a^m+1$ is not the sum of two squares. Then there
exists some prime $q\equiv 3\pmod{4}$ such that $q^j||a^m+1, \ j$
odd. Since $s=\frac{n}{m}$ is the sum of two squares, we know
$q^k||s, \ k$ even. Then $n=mq^ks',$ where $\gcd(s',q)=1,$ so
$q^{k+j}||a^n+1, \ k+j$ odd (Lemma ~\ref{general3}). This is a
contradiction because we assumed $a^n+1$ is the sum of two
squares. Therefore if $a^n+1$ is the sum of two squares for some odd $n$, then $a^m+1$ is also the sum of two squares.

Let $\delta \mid \frac{n}{m},$ where $\delta$ is the sum of two
squares, and suppose $a^n+1$ is the sum of two squares. We will show
that $a^{m \delta }+1$ is the sum of two squares. Assume by
contradiction that there exists a prime $q\equiv3\pmod{4} $ such that
$q^j\parallel a^{m \delta}+1, \ j$ odd. 

Since $\delta$ is the sum of two squares, we know $q^k\parallel \delta, \ k$ even, $k\geq0.$ Because $q$ must divide $a^{n} + 1$ to an even power,
Lemma~\ref{general3} implies that $q^l\parallel \frac{n}{m\delta}, \ l$ odd, so  $q^{l+k}\parallel \frac{n}{m},\ l+k$ odd, which is a contradiction because $\frac{n}{m}$ is the sum of two squares. Thus if $a^n+1$ is the sum of two squares, $a^{m\delta}+1$ is the sum of two squares for all $\delta\mid \frac{n}{m}$ such that $\delta$ is the sum of two squares.  

Finally, we will show that if $a^{np^2} + 1$ is the sum of two squares for
some $p \equiv 3 \pmod{4}$, then $p| a^n + 1$. By Lemma ~\ref{oddlem1} 
we know $a^{np}+1$ is not the sum of two squares, so there exists some 
$q\equiv 3\pmod{4}$  with  $q^j || a^{np} + 1, \ j $ odd.
If $q \not= p$, then by Lemma~\ref{general3} we have
$q^j || a^{np^2} + 1, \ j $ odd, which contradicts $a^{np^2}+1$ being the sum of two squares. Hence $q=p$, and since $p|a^{np}+1$ and $a^{np} \equiv a^n \pmod{p}$, 
we have that $p|a^n+1$, as desired. 
\end{proof}


We conclude this section with a heuristic giving evidence for Conjecture~\ref{bigconj}. Suppose first
that $a \equiv 0 \text{ or } 1$ mod $4$. In this case,
if $a^{n} + 1$ is the sum of two squares for any $n$,
then $a+1$ is the sum of two squares. Let $A_{p}$
be the event that $\Phi_{2p}(a)$ is the sum of two squares.
It seems plausible that the probability that this even
occurs is $\approx \frac{K}{\sqrt{\ln(\Phi_{2p}(a))}}
\approx \frac{K}{\sqrt{p}}$. Since $\sum_{p \equiv 1 \pmod{4}} \frac{1}{\sqrt{p}}$ diverges, we should
expect an infinite number of the events $A_{p}$ to occur,
and this would yield infinitely many primes $p$ for which
$a^{p} + 1$ is the sum of two squares.

If $a \equiv 2 \pmod{4}$, then Theorem~\ref{even} implies
there is at most one $n$ so that $a^{n} + 1$ is the sum
of two squares.

In the case that $a \equiv 3 \pmod{4}$, let
$m$ denote the smallest positive integer so that
$\frac{a+1}{m}$ is the sum of two squares. First,
consider primes $p \equiv 1 \pmod{4}$ so that
$a^{mp} + 1$ is the sum of two squares. We have
\[
\frac{a^{mp}+1}{a^{m} + 1} = \prod_{\substack{d | 2mp \\ d \nmid 2m}} \Phi_{d}(a).
\]
Theorem~\ref{cyclothm} implies that if we write
$\Phi_{d}(a) = \gcd(\Phi_{d}(a),m) c_{d}$,
then the $c_{d}$ are pairwise coprime and this implies
that $c_{d}$ is the sum of two squares for all $d$.
It seems plausible that the $c_{d}$ being the sum of two
squares are independent, and so the probability that
$a^{mp} + 1$ is the sum of two squares is $\approx \prod_{d} \frac{1}{\sqrt{\ln(c_{d})}} \approx p^{-\tau(m)/2}$, where $\tau(m)$ is the number of divisors
of $m$. This sum diverges if $m = 1$ or $m$ is prime, and converges if $m$ is composite. In particular, in the case that $m$ is composite,
there are only finitely many primes $p$ so that $a^{mp} + 1$ is the sum of
two squares.

Then, Theorem~\ref{3mod4} then implies that there are only finitely
many primes that can divide some number $n$ so that $a^{n} + 1$ is the
sum of two squares. If there are infinitely many $n$ so that $a^{n} + 1$
is the sum of two squares, it follows then that there is a prime
$p$ so that $a^{p^{r}} + 1$ is the sum of two squares for infinitely many
$r$. We have that $a^{p^{r}} + 1 = \prod_{i=0}^{r} \Phi_{2p^{i}}(a)$.
If we write $r_{i} = \frac{\Phi_{2p^{i}}(a)}{\gcd(\Phi_{2p^{i}}(a),p)}$,
then Theorem~\ref{cyclothm} implies that $\gcd(r_{i},r_{j}) = 1$.
It follows from this that $r_{i}$ is the sum of two squares for all $i \geq 1$.
Assuming that these events are independent, the probability this
occurs is $\sum_{i} \frac{K}{\sqrt{\log(r_{i})}}$. But this
sum converges. Therefore the ``probability is zero'' that
there are infinitely many $n$ so that $a^{n} + 1$ is the sum of two
squares in the case when $a \equiv 3 \pmod{4}$ and $m$ is composite.

As an example, we consider $a = 4713575$, with a composite $m$ value
of $m=21$.  We conjecture that there are finitely many $n$ so that
$a^{n} + 1$ is the sum of two squares. So far, we know only of $n=21$
and $n=105$. 


\section{$p\equiv1\pmod{4}$}\label{p1mod4}
The previous theorems put constraints on when $a^n+1$ can be the sum of two squares for different categories of $a.$ The following proof of Theorem ~\ref{px2} uses Aurifeuillian factorization to show that when $a=pv^2,$ where $p\equiv 1 \pmod{4}$ is a prime and  $p\nmid v,$ there are either zero or infinitely many odd integers $n$ such that $a^n+1$ is the sum of two squares. 

\begin{proof}[Proof of Theorem ~\ref{px2}]
Let $a=pv^2,$ where $p\equiv 1 \pmod{4}$ is prime and $p\nmid v.$ Suppose $a^n+1$ is the sum of two squares and consider:
\begin{align*}
    a^{np}+1=\prod_{\delta|n}\Phi_{2\delta}(a)\prod_{\substack{\delta |np \\ \delta \nmid n}}\Phi_{2\delta}(a).
\end{align*}

We know $\displaystyle\prod_{\delta|n}\Phi_{2\delta}(a)=a^n+1$ is the sum of two squares. Consider the Aurifeuillian factorization of $\Phi_{2\delta}(a),$ where $\delta|np, \delta \nmid n,$ $x=-kv^2,\ k=-p\equiv 3 \pmod{4},$ and $q$ is odd:
\begin{align*}
    \Phi_{2\delta }(x)&=\big(F(x)\big)^2-kx^q\big(G(x)\big)^2\\
    \Phi_{2\delta }(-kv^2)&=\big(F(-kv^2)\big)^2-k(-kv^2)^q\big(G(-kv^2)\big)^2\\
    &=\big(F(-kv^2)\big)^2+k^{q+1}v^{2q}\big(G(-kv^2)\big)^2\\
    &=\big(F(-kv^2)\big)^2+\big(k^{\frac{q+1}{2}}v^{q}G(-kv^2)\big)^2\\
    &=\Phi_{2\delta}(a).
\end{align*}
Therefore $\Phi_{2\delta}(a)$ is the sum of two squares for any $\delta |n.$ Thus $a^{np}+1$ is the sum of two squares.
Conversely, suppose that $a^{np}+1$ is the sum of two squares. Then we can see again that $\Phi_{2\delta p}(a)$ is the sum of two squares for any factor $\delta.$ This implies that $\displaystyle \prod_{\delta |n}\Phi_{2\delta }(a)=a^n+1$ is the sum of two squares.
\end{proof}

Now, we will construct an infinite family of
number $a = f(X)$ so that $a^{p} + 1$ is the sum of
two squares.

\begin{proof}[Proof of Theorem~\ref{poly}]
Suppose $p\equiv 1\pmod{4}$, then there exists an even integer $u$  and an odd integer $v$ such that $p=u^{2}+v^{2}$. Then consider the following polynomials:
\begin{eqnarray*}
A(X)& =& \frac{u}{2}pX^{2}+vX,\\
B(X)& =& \frac{u^{2}}{2}pX^{2}-1, \text{ and} \\
C(X)& =& \frac{uv}{2}pX^{2}+pX.
\end{eqnarray*}

Let $f(X)=pA(X)^{2}$, then we have
\begin{eqnarray*}
f(X)^{p}+1 &=& (f(X)+1)\Phi_{2p}(f(X)) \\
&=&(pA(X)^{2}+1)\Phi_{2p}(pA(X)^{2}).
\end{eqnarray*}

It is straightforward to check that $f(X)+1$ can be written as the sum of two squares: $pA(X)^2+1=B(X)^2+C(X)^2.$ 
Then consider the Aurifeuillian factorization of 
$\Phi_{2p}(x)$, where we let $k=-p$ and $x=pA(X)^2$, then we get the following: 
\begin{eqnarray*}
\Phi_{2p}(x)&=&F(x)^{2}-kxG(x)^{2}\\
\Phi_{2p}\left(pA(X)^{2}\right)&=&
\left(F\left(pA(X)^2\right)\right)^{2}
-p\left(-pA(X)^2\right)
\left(G\left(pA(X)^2\right)\right)^{2}\\
&=&\left(F\left(pA(X)^2\right)\right)^{2}+
\left(p^{2}A(X)^2\right)\left(\left( G(pA(X)^2\right)\right)^{2}\\
&=&\left(F\left(pA(X)^2\right)\right)^{2}+
\left(pA(X)\left(G\left(pA(X)^2\right)\right)\right)^{2}\\
&=&F(x)^{2}+\left(-kxG(x)\right)^{2}
\end{eqnarray*}

Therefore, $\Phi_{2p}\left(f(X)\right)$ can be written as the sum of two squares as well. This implies that $f(X)^{p}+1$ is the product of two terms, each of which can be written as the sum of two squares. 
\end{proof}

\newpage
\section{Chart}
Here is a chart that illustrates the first few odd integers $n$ such that $a^n+1$ is the sum of two squares for all integers $a\in[1,50].$

\begin{center}
 \begin{tabular}{||c c c || c c c||} 
 \hline
 \ \ a \ \ \ \ \  & n \ \ \ \ \ &  Property & \ \ a \ \ \ \ \  & n \ \ \ \ \ &  Property\\ [1.2 ex] 
 \hline  
  \ \ 1 \ \ \ \ \  & all \ \ \ \ \  & Thm ~\ref{perfect} & \ \ 26 \ \ \ \ \  & -  \ \ \ \ \ & Thm ~\ref{even}\\ [1.2 ex]              

 \ \ 2 \ \ \ \ \  & 3 \ \ \ \ \  & Thm ~\ref{even} & \ \ 27 \ \ \ \ \  & -  \ \ \ \ \ & Thm ~\ref{3mod4} \\ [1.2 ex]              

 \ \ 3 \ \ \ \ \  & 1, 5, 13, 65,\ldots \ \ \ \ \  & Thm ~\ref{3mod4} & \ \ 28 \ \ \ \ \  & 1, 3, 11, 19,\ldots  \ \ \ \ \ & Thm ~\ref{even}\\ [1.2 ex]              

 \ \ 4 \ \ \ \ \  & all \ \ \ \ \  & Thm ~\ref{perfect} & \ \ 29 \ \ \ \ \  & -  \ \ \ \ \ & Thm ~\ref{5} \\ [1.2 ex]              
   
 \ \ 5 \ \ \ \ \  & -  \ \ \ \ \ & Thm ~\ref{5} & \ \ 30 \ \ \ \ \  & 31  \ \ \ \ \ & Thm ~\ref{even}\\ [1.2 ex]              
   
  \ \ 6 \ \ \ \ \  & 7  \ \ \ \ \ & Thm ~\ref{even} & \ \ 31 \ \ \ \ \  & 1, 5, 25, 41,\ldots  \ \ \ \ \ & Thm ~\ref{3mod4}\\ [1.2 ex]              
   
   \ \ 7 \ \ \ \ \  & 1, 13, 17, 29,\ldots  \ \ \ \ \ & Thm ~\ref{3mod4} & \ \ 32 \ \ \ \ \  & -  \ \ \ \ \ & Thm ~\ref{even}\\ [1.2 ex]              
   
   \ \ 8 \ \ \ \ \  & 1  \ \ \ \ \ & Thm ~\ref{even} & \ \ 33 \ \ \ \ \  & 1, 5, 7, 17,\ldots  \ \ \ \ \ & Thm ~\ref{1mod8} \\ [1.2 ex]              
   
   \ \ 9 \ \ \ \ \  & all  \ \ \ \ \ & Thm ~\ref{perfect} & \ \ 34 \ \ \ \ \  & -  \ \ \ \ \ & Thm ~\ref{even}\\ [1.2 ex]              
   
   \ \ 10 \ \ \ \ \  & -  \ \ \ \ \ & Thm ~\ref{even} & \ \ 35 \ \ \ \ \  & 1, 9, 13, 29,\ldots  \ \ \ \ \ & Thm ~\ref{3mod4}\\ [1.2 ex]              
   
   \ \ 11 \ \ \ \ \  & 3, 159,\ldots  \ \ \ \ \ & Thm ~\ref{3mod4} &\ \ 36 \ \ \ \ \  & all  \ \ \ \ \ & Thm ~\ref{perfect}\\ [1.2 ex]              
   
   \ \ 12 \ \ \ \ \  & 1, 5, 11, 23,\ldots  \ \ \ \ \ & Thm ~\ref{even} & \ \ 37 \ \ \ \ \  & -  \ \ \ \ \ & Thm ~\ref{5}\\ [1.2 ex]              
   
   \ \ 13 \ \ \ \ \  & -  \ \ \ \ \ & Thm ~\ref{5} & \ \ 38 \ \ \ \ \  & -  \ \ \ \ \ & Thm ~\ref{even}\\ [1.2 ex]              
   
    \ \ 14 \ \ \ \ \  & 3  \ \ \ \ \ & Thm ~\ref{even} & \ \ 39 \ \ \ \ \  & 1, 13, 37, 61,\ldots  \ \ \ \ \ & Thm ~\ref{3mod4}\\ [1.2 ex]              
   
    \ \ 15 \ \ \ \ \  & 1, 29, 89, 97,\ldots  \ \ \ \ \ & Thm ~\ref{3mod4} & \ \ 40 \ \ \ \ \  & 1, 5, 13, 53,\ldots  \ \ \ \ \ & Thm ~\ref{even}\\ [1.2 ex]              
   
    \ \ 16 \ \ \ \ \  & all  \ \ \ \ \ & Thm ~\ref{perfect} & \ \ 41 \ \ \ \ \  & -  \ \ \ \ \ & Thm ~\ref{1mod8}\\ [1.2 ex]              
   
    \ \ 17 \ \ \ \ \  & 1, 7, 17, 23,\ldots  \ \ \ \ \ & Thm ~\ref{px2} & \ \ 42 \ \ \ \ \  & -  \ \ \ \ \ & Thm ~\ref{even}\\ [1.2 ex]              
   
    \ \ 18 \ \ \ \ \  & 19  \ \ \ \ \ & Thm ~\ref{even} & \ \ 43 \ \ \ \ \  & -  \ \ \ \ \ & Thm ~\ref{3mod4}\\ [1.2 ex]              
   
    \ \ 19 \ \ \ \ \  & 1, 17, 29, 37,\ldots  \ \ \ \ \ & Thm ~\ref{3mod4} & \ \ 44 \ \ \ \ \  & 1, 5, 7, 17,\ldots  \ \ \ \ \ & Thm ~\ref{even}\\ [1.2 ex]              
   
    \ \ 20 \ \ \ \ \  & -  \ \ \ \ \ & Thm ~\ref{even} & \ \ 45 \ \ \ \ \  & -  \ \ \ \ \ & Thm ~\ref{5}\\ [1.2 ex]              
   
    \ \ 21 \ \ \ \ \  & -  \ \ \ \ \ & Thm ~\ref{5} & \ \ 46 \ \ \ \ \  & -  \ \ \ \ \ & Thm ~\ref{even}\\ [1.2 ex]              
   
    \ \ 22 \ \ \ \ \  & -  \ \ \ \ \ & Thm ~\ref{even} &  \ \ 47 \ \ \ \ \  & -  \ \ \ \ \ & Thm ~\ref{3mod4}\\ [1.2 ex]              
   
    \ \ 23 \ \ \ \ \  & 3, 123,\ldots  \ \ \ \ \ & Thm ~\ref{3mod4} & \ \ 48 \ \ \ \ \  & 1, 3, 5, 17,\ldots  \ \ \ \ \ & Thm ~\ref{even}\\ [1.2 ex]              
   
    \ \ 24 \ \ \ \ \  &  1, 7, 11, 19,\ldots \ \ \ \ \ & Thm ~\ref{even} & \ \ 49 \ \ \ \ \  & all  \ \ \ \ \ & Thm ~\ref{perfect}\\ [1.2 ex]              
   
    \ \ 25 \ \ \ \ \  & all  \ \ \ \ \ & Thm ~\ref{perfect} & \ \ 50 \ \ \ \ \  & -  \ \ \ \ \ & Thm ~\ref{even}\\ [1ex] 
    \hline

 \end{tabular}
\end{center}
\newpage

\bibliographystyle{plain}
\bibliography{refs}

\end{document}